\documentclass[12pt,reqno]{amsart}
\usepackage[british]{babel}

\usepackage{graphicx}
\usepackage{amsmath, amssymb, mathtools,cite} 
\usepackage{geometry, url, flafter}
\usepackage{microtype}
\usepackage{xcolor}
\newcommand{\pqbinom}[3]{\binom{#1}{#2}_{\! #3}}

\newcommand{\ud}{\,\mathrm{d}}
\DeclareMathOperator{\Real}{Re}
\DeclareMathOperator{\Imag}{Im}

\DeclareMathOperator{\Li}{Li}
\DeclareMathOperator{\Arg}{Arg}
\DeclareMathOperator{\csch}{csch}

\newtheorem{theorem}{Theorem}[section]
\newtheorem{proposition}[theorem]{Proposition}
\newtheorem{lemma}[theorem]{Lemma}
\newtheorem{corollary}[theorem]{Corollary}
\theoremstyle{definition}
\newtheorem{definition}[theorem]{Definition}

\newtheorem*{theoremA}{Theorem A}
\newtheorem*{theoremB}{Theorem B}

\theoremstyle{remark}
\newtheorem*{remark}{Remark}
\numberwithin{equation}{section}

\title{On $(p,q)$-binomial coefficient ratios for complex parameters}

\begin{document}
\author{P. {\AA}hag} \address{Department
  of Mathematics and Mathematical Statistics, Ume{\aa} University,
  SE-901 87 Ume{\aa}, Sweden}
 \email{per.ahag@umu.se}

\author{R. Czy{\.z}}\address{Faculty of Mathematics and Computer
  Science, Jagiellonian University, \L ojasiewicza 6, 30-348 Krak\'ow,
  Poland}\email{rafal.czyz@im.uj.edu.pl}

\author{P. H. Lundow}
\address{Department of Mathematics and Mathematical Statistics,
  Ume{\aa} University, SE-901 87 Ume{\aa}, Sweden}
\email{per-hakan.lundow@umu.se}

\begin{abstract}
We prove local asymptotics for near-central complex $(p,q)$-binomial
coefficient modulus ratios allowing an imaginary parameter
perturbation of order $n^{-3/4}$ at a $\sqrt{n}$ length scale from
the centre. Moreover, we obtain ratio asymptotics for a smaller imaginary
perturbation of order $n^{-5/4}$ at the length scale $n^{3/4}$.
These results were obtained by reducing the two-parameter coefficients
to just one parameter, giving a branch-free logarithmic representation
of the second-order ratio and, hence, uniform complex curvature
asymptotics for near-central ratios.
\end{abstract}



\maketitle

\section{Introduction}\label{sec:intro}

For integers $0\le k\le n$ and complex parameters $p,q$ with
$p\ne0$ and
\[
p^j\ne q^j\qquad (1\le j\le n),
\]
the $(p,q)$-binomial coefficients are defined by
\begin{equation}\label{eq:intro-pq-def}
\pqbinom{n}{k}{p,q}
:=
\prod_{j=1}^{k}\frac{p^{\,n-k+j}-q^{\,n-k+j}}{p^j-q^j}.
\end{equation}
The empty product is one.  The non-vanishing condition is used only to
make \eqref{eq:intro-pq-def} literal.  It is automatic for the
asymptotic choices considered below.  For $p=1$, one obtains the usual
$q$-binomial coefficients, or Gaussian polynomials, after the standard
polynomial cancellation \cite{GasperRahman2004,DLMF17,StanleyEC1}.  If
$q$ is a prime power, then $\pqbinom{n}{k}{q}$ counts the
$k$-dimensional subspaces of $\mathbb F_q^n$ \cite{StanleyEC1}.
Equivalently, the coefficient of $q^\ell$ in $\pqbinom{n}{k}{q}$ counts
the $k$-element subsets of $\{1,\ldots,n\}$ with element sum
$\binom{k+1}{2}+\ell$.  The unimodality and strict unimodality of
these coefficient sequences have been studied from several different
viewpoints \cite{Dhand2014,OHara1990,PakPanova2013}.  The
$(p,q)$-binomial coefficient is the homogeneous symmetric form of the
same polynomial.

The polynomial
\[
P_n(w;p,q):=\sum_{k=0}^{n}\pqbinom{n}{k}{p,q}w^k
\]
is the $(p,q)$-Rogers--Szeg\H{o} polynomial. Its one-parameter form is
\begin{equation}\label{eq:intro-RS-polynomial}
P_n(w;q):=P_n(w;1,q)=\sum_{k=0}^{n}\pqbinom{n}{k}{q}w^k.
\end{equation}

Szeg\H{o}'s 1926 note~\cite{Szego1926} relates finite polynomials with
$q$-binomial coefficients to theta functions and orthogonality on the unit
circle and connects them to Rogers' earlier identities. Atakishiyev
and Nagiyev
studied an orthogonality relation on the real line and related the resulting
functions to Stieltjes--Wigert functions by the Fourier transform
\cite{AtakishiyevNagiyev1994}.

Lubinsky and Saff~\cite{LubinskySaff1987} used Rogers--Szeg\H{o} polynomials
to analyse Pad\'e approximants of partial theta functions and the zero
distribution of these polynomials on the unit circle. Their work is global in
the variable $z$ and assumes fixed $|q|=1$. Here $q/p$ depends on $n$,
approaches $1$, and the estimates concern coefficients near the centre.

Corcino~\cite{Corcino2008} studied algebraic properties of
$(p,q)$-binomial coefficients, including recurrences, generating functions,
inverse relations, and orthogonality relations. Jagannathan and Sridhar gave
the corresponding $(p,q)$-Rogers--Szeg\H{o} polynomial and its oscillator
interpretation \cite{JagannathanSridhar2010}.

Positive $(p,q)$-binomial distributions were later used by Lundow and
Rosengren~\cite{LundowRosengren2010,LundowRosengren2013}, after determining
asymptotic properties of the coefficients, to approximate Ising-model
magnetisation distributions. In the mean-field case this distribution is
obtained exactly in the limit $q/p\to1$, with a simple relation between the
temperature and $p,q$. Moreover, $(p,q)$-binomial distributions approximate
the magnetisation distribution of high-dimensional systems ($d\ge5$) very
accurately, with the correct limiting shapes, although the relation between
temperature and the $(p,q)$-parameters is less clear.

The positive-parameter unimodal or bimodal structure was proved by Su and
Wang \cite{SuWang2012}. The real transition function relating $p,q$ to the
locations of the distribution maxima was studied by the present authors in
\cite{AhagCzyzLundow2024}, and complex branches of the associated
generalised Lambert $W$-type function were studied in
\cite{AhagCzyzLundow2025}. These papers lead directly to the present complex
coefficient problem.

Arabi Ardehali and Rosengren~\cite{ArabiArdehaliRosengren2026} obtained
uniform asymptotic expansions for complex $q$-shifted factorials. For
$(e^{-y};e^{-t})_\infty$, their Corollary~2 assumes that $\Arg t$ is fixed
in $(-\pi/2,\pi/2)$. In Theorem~A below, $t=v/n+i\delta/n^{3/4}$ with
$v>0$. For fixed $\delta\ne0$,
\[
\frac{|\Imag t|}{\Real t}=\frac{|\delta|}{v}n^{1/4}\longrightarrow\infty.
\]
Thus $t$ approaches zero tangentially to the imaginary axis, and their
fixed-angle remainder estimate does not directly apply along these paths.
Their exact product identity~\cite[Eq.~(2)]{ArabiArdehaliRosengren2026}
and its higher-order form in Corollary~1 hold for $\Real t>0$ and
$y\notin2\pi i\mathbb Z$, with the branch conventions specified there,
without a fixed-angle assumption. The
logarithmic and dilogarithmic terms below have the usual form of
$q$-product expansions. Here the estimates are uniform on the parameter
sets in Definition~\ref{def:admissible-set} and the stated coefficient
ranges.

The Ising-model and transition-function references explain the origin of the
problem but are not used in the proofs. We separate the two parameters and
represent the second-order ratio by an absolutely convergent logarithmic
series, without choosing a logarithm branch. Estimates for its real part
then give the logarithmic modulus ratios by summation of second-order
differences.

Theorem~\ref{thm:complex-middle} approximates the logarithms of the
centred modulus ratios for $|\ell|\le Ln^{3/4}$, with uniform remainder
$O_{\mathcal D,L}(n^{-9/4})$. The approximation retains logarithmic and
dilogarithmic terms rather than expanding them into a quartic polynomial.

Set
\[
r:=\frac qp.
\]
Then, by Proposition~\ref{prop:pq-reduction}, for every $0\le k\le n$,
\[
\pqbinom{n}{k}{p,q}=p^{\,k(n-k)}\pqbinom{n}{k}{r}.
\]
Thus every local question splits into the explicit quadratic prefactor
$p^{k(n-k)}$ and the one-parameter coefficients
\begin{equation}\label{eq:intro-r-binomial}
\pqbinom{n}{k}{r}
:=\prod_{j=1}^{k}\frac{1-r^{\,n-k+j}}{1-r^j}.
\end{equation}
For $0\le k\le n-1$ put
\begin{equation}\label{eq:intro-rho}
\rho_k(r):=
\frac{\pqbinom{n}{k+1}{r}}{\pqbinom{n}{k}{r}}
=\frac{1-r^{n-k}}{1-r^{k+1}},
\end{equation}
and for $1\le k\le n-1$ put
\begin{equation}\label{eq:intro-R}
R(n,k;r):=
\frac{\pqbinom{n}{k-1}{r}\,\pqbinom{n}{k+1}{r}}
     {\pqbinom{n}{k}{r}^{2}}
=\frac{\rho_k(r)}{\rho_{k-1}(r)}
=\frac{(1-r^{n-k})(1-r^k)}{(1-r^{n-k+1})(1-r^{k+1})}.
\end{equation}
For the full $(p,q)$-coefficients we use the corresponding notation
\begin{equation}\label{eq:intro-full-ratios}
\rho_k(p,q):=\frac{\pqbinom{n}{k+1}{p,q}}{\pqbinom{n}{k}{p,q}},
\qquad
R(n,k;p,q):=\frac{\rho_k(p,q)}{\rho_{k-1}(p,q)}.
\end{equation}
The reduction gives
\begin{equation}\label{eq:intro-ratio-reduction}
\rho_k(p,q)=p^{\,n-2k-1}\rho_k(r),
\qquad
R(n,k;p,q)=p^{-2}R(n,k;r).
\end{equation}
Since complex coefficients have no natural order, we work with the modulus and
its discrete second-order difference.  Define
\begin{equation}\label{eq:intro-g}
g(k):=\log\left|p^{\,k(n-k)}\pqbinom{n}{k}{r}\right|
=\log\left|\pqbinom{n}{k}{p,q}\right|,
\qquad 0\le k\le n.
\end{equation}
Then, for $1\le k\le n-1$,
\begin{equation}\label{eq:intro-second-diff}
\Delta^2g(k):=g(k+1)-2g(k)+g(k-1)=\log|R(n,k;r)|-2\log|p|.
\end{equation}
If $n$ is even, $\kappa:=n/2$, and $0\le\ell\le\kappa$, then
\begin{equation}\label{eq:intro-middle-product}
\frac{\pqbinom{n}{\kappa+\ell}{r}}{\pqbinom{n}{\kappa}{r}}
=
\prod_{j=0}^{\ell-1}\rho_{\kappa+j}(r),
\end{equation}
with negative $\ell$ handled by the coefficient symmetry.

For the main complex estimates we use $n$-dependent parameters of the form
\begin{equation}\label{eq:intro-parameters}
p=\exp\!\left(-\frac{u}{n}-i\,\frac{\beta}{n^{3/4}}\right),
\qquad
q=\exp\!\left(-\frac{z}{n}-i\,\frac{\alpha}{n^{3/4}}\right),
\end{equation}
with fixed real parameters $u,z,\alpha,\beta$ satisfying $z-u>0$.  Hence
\begin{equation}\label{eq:intro-r-exponent}
r=\frac qp=\exp(-t),
\qquad
t=\frac{v}{n}+i\,\frac{\delta}{n^{3/4}},
\qquad
v:=z-u>0,
\qquad
\delta:=\alpha-\beta.
\end{equation}

\begin{definition}\label{def:admissible-set}
Fix constants
\[
0<v_{\min}\le v_{\max}<\infty,
\qquad
0\le \delta_{\max}<\infty.
\]
The admissible set for the exponent parameters $(v,\delta)$ is
\[
\mathcal D:=\{(v,\delta)\in\mathbb R^2:
v_{\min}\le v\le v_{\max},\ |\delta|\le \delta_{\max}\}.
\]
\end{definition}
All constants implicit in $O_{\mathcal D,L}(\cdot)$ are independent of $n$ and of
the chosen $(v,\delta)\in\mathcal D$.  For $z>0$ write
\begin{equation}\label{eq:intro-AB}
A(z):=z\left(\coth\frac z4-1\right),
\qquad
B(z):=\frac{z^3\sinh(z/2)}{8\sinh^4(z/4)}.
\end{equation}
Then $B(z)>0$.

On the positive real axis, the choice $2u=A(v)$ cancels the leading
quadratic term, and the first non-zero term is quartic for
$\ell=O(n^{3/4})$. For the complex parameters in
\eqref{eq:intro-r-exponent}, the quantity $\ell\Imag t$ is bounded on
this coefficient range. However,
\[
nt=v+i\delta n^{1/4},
\]
does not converge when $\delta\ne0$. In this case, Theorem~A gives
an expansion with an $n$-dependent quadratic coefficient on the smaller
range $|\ell|\le L\sqrt n$, rather than convergence to a fixed Gaussian.

The exponent $5/4$ in Theorem~B comes from a different calculation. Put
$t=z/n+i\alpha n^{-a}$ with $z>0$, $\alpha\in\mathbb R$, and $a>1$,
and let $\mathcal A$ be the analytic extension of $A$ to a neighbourhood
of $z$. Then
\[
\Real\mathcal A(nt)
=A(z)-\frac{\alpha^2}{2}A''(z)n^{2(1-a)}
+O_{z,\alpha}\!\left(n^{4(1-a)}\right).
\]
As in the proof of Proposition~\ref{prop:small-imaginary-perturbation},
this expression enters the quadratic term of the logarithmic modulus ratio
through $-\ell^2\Real\mathcal A(nt)/(2n)$. At $\ell=xn^{3/4}$, the
first correction due to the imaginary part is
\[
\frac{\alpha^2}{4}A''(z)x^2 n^{5/2-2a}.
\]
The power of $n$ is zero when $a=5/4$. Theorem~B proves that this correction
remains in the limit and that the quartic coefficient is unchanged.
In this case $\ell\Imag t=O(n^{-1/2})$, so its convergence to zero does
not imply that the imaginary perturbation disappears from the leading
modulus asymptotic.

\begin{theoremA}
Assume that $n$ is even and write $\kappa:=n/2$. Let $p,q,r,t$ be
given by \eqref{eq:intro-parameters}--\eqref{eq:intro-r-exponent} for fixed real
$u,z,\alpha,\beta$ with $v=z-u>0$. Put
\[
\tau:=v+i\delta n^{1/4},
\qquad
w:=\frac{\tau}{2},
\]
and define
\[
F(w):=\frac{1}{e^w-1},
\qquad
G(w):=\frac{e^w}{(e^w-1)^2},
\qquad
H(w):=\frac{e^w(e^w+1)}{(e^w-1)^3}.
\]
For every fixed $L>0$, uniformly for all integers $\ell$ with
$|\ell|\le L\sqrt n$, and with $x:=\ell/\sqrt n$,
\begin{align}
\log\left|
\frac{\pqbinom{n}{\kappa+\ell}{p,q}}
     {\pqbinom{n}{\kappa}{p,q}}
\right|
&=
x^2\left(u-C_n(v,\delta)\right)
+\frac{x^2}{2n}\Real\left(\tau^2G(w)\right)
-\frac{x^4}{12n}\Real\left(\tau^3H(w)\right)
\nonumber\\
&\quad+O_{v,\delta,L}(\mathcal E_n),
\label{eq:main-sqrt-pq}\\[1mm]
\log\left|
\frac{\pqbinom{n}{\kappa+\ell}{r}}{\pqbinom{n}{\kappa}{r}}
\right|
&=
-x^2 C_n(v,\delta)
+\frac{x^2}{2n}\Real\left(\tau^2G(w)\right)
-\frac{x^4}{12n}\Real\left(\tau^3H(w)\right)
\nonumber\\
&\quad+O_{v,\delta,L}(\mathcal E_n),
\label{eq:main-sqrt-r}
\end{align}
where
\[
\mathcal E_n:=n^3|t|^5+n^2|t|^4+n|t|^3
\]
and
\begin{equation}\label{eq:main-Cn}
C_n(v,\delta)
:=\Real\left(\tau F(w)\right)
=\Real\left(
\frac{v+i\delta n^{1/4}}
{e^{v/2+i\delta n^{1/4}/2}-1}
\right).
\end{equation}
If $\delta=0$, then $C_n(v,0)=A(v)/2$. If $\delta\ne0$ and
$\vartheta:=\delta n^{1/4}/2$, then
\begin{equation}\label{eq:main-Cn-explicit}
C_n(v,\delta)=
\frac{v(e^{v/2}\cos\vartheta-1)
+\delta n^{1/4}e^{v/2}\sin\vartheta}
{e^{v}-2e^{v/2}\cos\vartheta+1}.
\end{equation}
\end{theoremA}

\begin{remark}
The constant $L$ in Theorem A is fixed as $n\to\infty$. Thus the theorem is
uniform for bounded $x=\ell/\sqrt n$ and gives a local expansion about the
central coefficient. It is not intended to approximate the entire coefficient
array or to locate maxima that leave this range as $n$ increases. The numerical
illustrations below show both the accuracy near the centre and the deterioration
farther away from it.
\end{remark}

\begin{theoremB}
Let $z>0$, let $\alpha,\gamma\in\mathbb R$, and assume that $n$ is even. Put
$\kappa:=n/2$ and
\[
t:=\frac zn+i\frac{\alpha}{n^{5/4}},
\qquad
r:=e^{-t},
\qquad
u:=\frac{A(z)}2+\frac{\gamma}{\sqrt n},
\]
\[
p:=e^{-u/n},
\qquad
q:=p \,r.
\]
For every fixed $L>0$, uniformly for all integer sequences $\ell_n$ such that
$x_n:=\ell_n/n^{3/4}$ satisfies $|x_n|\le L$,
\begin{equation}\label{eq:intro-small-imaginary-perturbation}
\left|
\frac{\pqbinom{n}{\kappa+\ell_n}{p,q}}
     {\pqbinom{n}{\kappa}{p,q}}
\right|
=
\exp\!\left(
\left(\gamma+\frac{\alpha^2}{4}A''(z)\right)x_n^2
-\frac{B(z)}{12}x_n^4
+O_{z,\alpha,\gamma,L}(n^{-1/2})
\right).
\end{equation}
The same formula holds with $\kappa-\ell_n$ in place of
$\kappa+\ell_n$. In particular, if $x_n\to x$, then
\[
\left|
\frac{\pqbinom{n}{\kappa\pm\ell_n}{p,q}}
     {\pqbinom{n}{\kappa}{p,q}}
\right|
\longrightarrow
\exp\!\left(
\left(\gamma+\frac{\alpha^2}{4}A''(z)\right)x^2
-\frac{B(z)}{12}x^4
\right).
\]
For $\gamma=0$, the imaginary perturbation contributes the quadratic term
$\alpha^2A''(z)x^2/4$ at the real critical value $2u=A(z)$.
\end{theoremB}

\begin{figure}[tbp]
  \begin{minipage}{0.45\textwidth}
    \centering
    \includegraphics[width=0.9\textwidth]{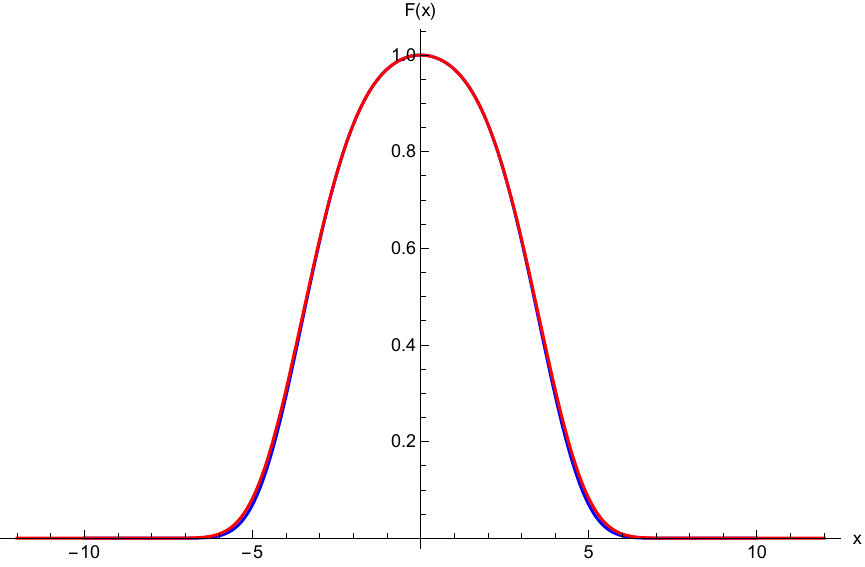}
  \end{minipage}%
  \begin{minipage}{0.45\textwidth}
    \centering
    \includegraphics[width=0.9\textwidth]{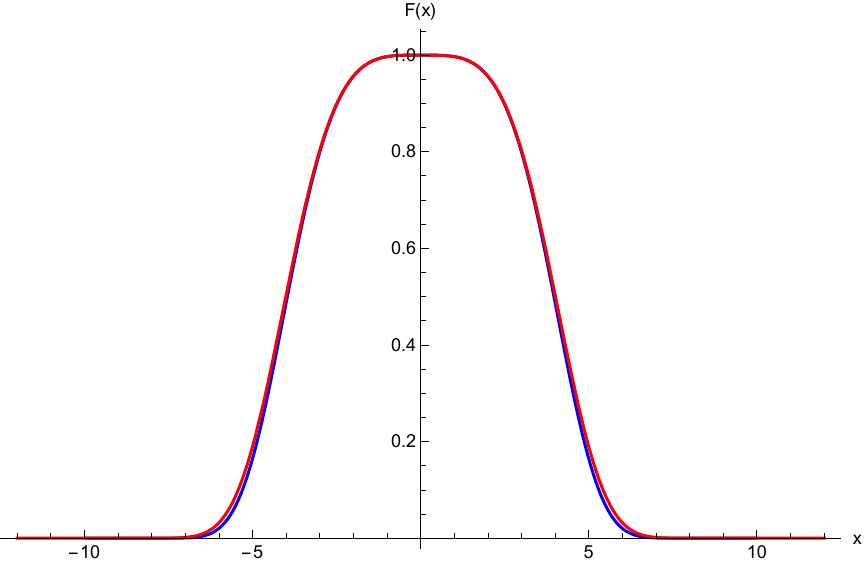}
  \end{minipage}
  \caption{Numerical illustration of Theorem A with $u=1$, $z=2$,
    $v=1$, $\delta=3/4$, and $x=\ell/\sqrt n$. In each panel, the
    blue curve is the modulus of the centred $(p,q)$-binomial coefficient
    ratio and the red curve is the approximation in
    \eqref{eq:main-sqrt-pq}. Left: $n=450$, with approximation
    $\exp(-0.0269x^2-0.00292x^4)$. Right: $n=496$, with approximation
    $\exp(-0.000590x^2-0.00264x^4)$, which is nearly quartic.}
  \label{fig:theoremA-central}
\end{figure}

\begin{figure}[tbp]
  \begin{minipage}{0.45\textwidth}
    \centering
    \includegraphics[width=0.9\textwidth]{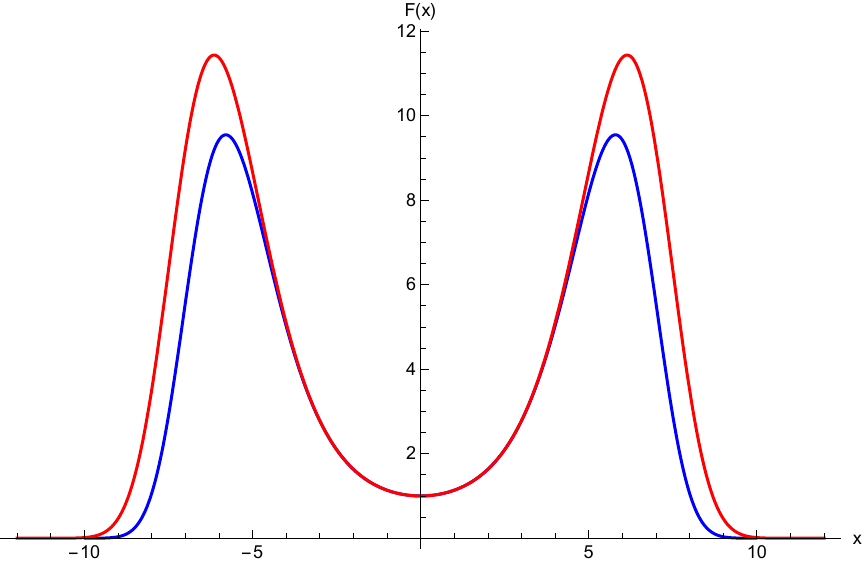}
  \end{minipage}%
  \begin{minipage}{0.45\textwidth}
    \centering
    \includegraphics[width=0.9\textwidth]{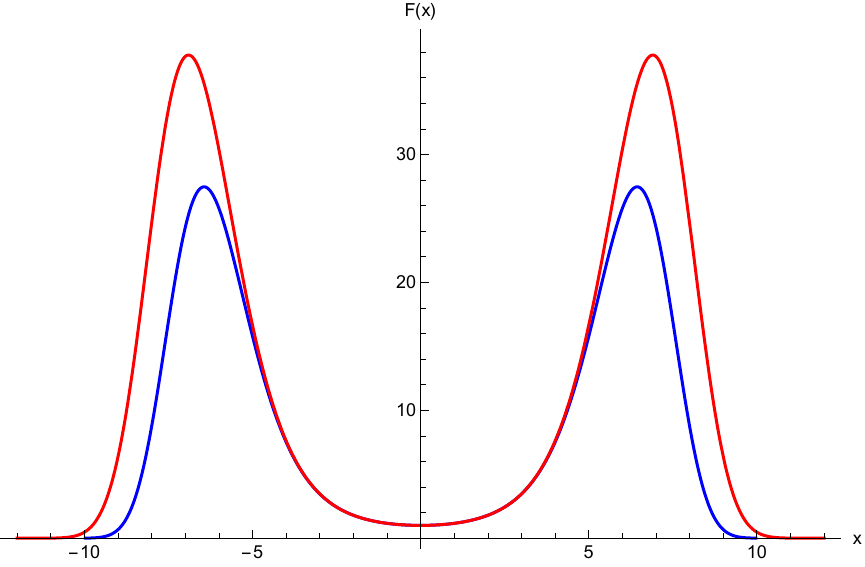}
  \end{minipage}
  \caption{The same comparison as in Figure~\ref{fig:theoremA-central},
    farther into the bimodal regime. Left: $n=750$, with approximation
    $\exp(0.129x^2-0.00171x^4)$. Right: $n=800$, with approximation
    $\exp(0.152x^2-0.00160x^4)$. The increasing discrepancy away from the
    centre is consistent with the local range stated in Theorem A.}
  \label{fig:theoremA-bimodal}
\end{figure}

\begin{figure}[tbp]
  \begin{minipage}{0.45\textwidth}
    \centering
    \includegraphics[width=0.9\textwidth]{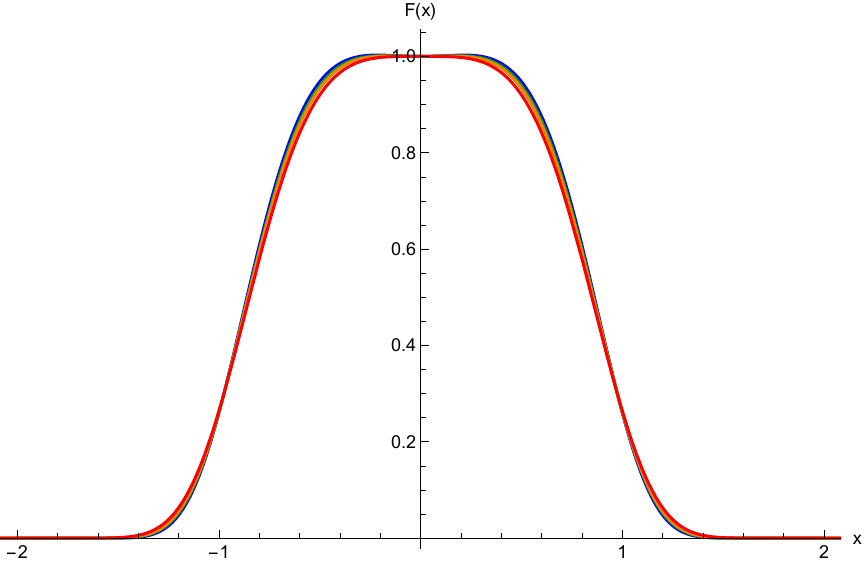}
  \end{minipage}%
  \begin{minipage}{0.45\textwidth}
    \centering
    \includegraphics[width=0.9\textwidth]{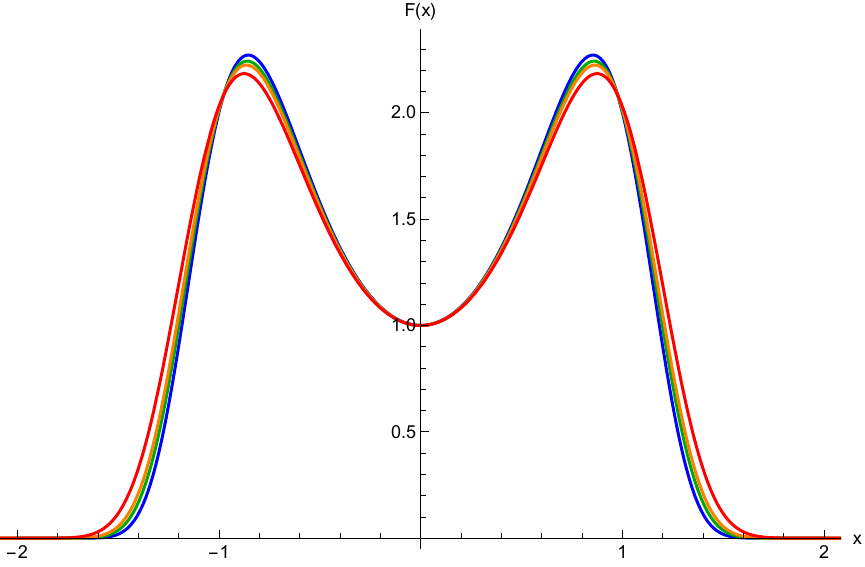}
  \end{minipage}
  \caption{Numerical illustration of Theorem B with
    $x=\ell/n^{3/4}$. The blue, green, and orange curves are the moduli of
    the centred $(p,q)$-binomial coefficient ratios for $n=256$, $512$,
    and $1024$, respectively; the red curve is the limiting function.
    Left: $z=\alpha=1$ and
    $\gamma=-\alpha^2A''(z)/4\approx-0.0406$, giving the limit
    $\exp(-1.33x^4)$. Right: $z=\alpha=1$ and $\gamma=2$, giving the
    limit $\exp(2.04x^2-1.33x^4)$.}
  \label{fig:theoremB-limit}
\end{figure}

The paper is organised as follows.
Section~\ref{sec:one-parameter} proves the one-parameter local formulae,
including the absolutely convergent logarithmic series for the second-order ratio. Section~\ref{sec:reduction} collects the factorisation, coefficient
symmetry, and ratio identities for the full $(p,q)$-coefficients. Section~\ref{sec:complex-central} contains the complex estimates near the
centre, including the result used in Theorem A.
Section~\ref{sec:small-imaginary-perturbation} treats the imaginary
perturbation of order $n^{-5/4}$ and proves Theorem B.
Section~\ref{sec:conclusion} gives further problems.

\section{The one-parameter local formulae}\label{sec:one-parameter}

This section proves the one-parameter formulae used below.
Lemma~\ref{lem:R-hyperbolic} gives the hyperbolic
form of the second-order ratio. Proposition~\ref{prop:Lambda} gives the logarithmic
series for the second-order ratio; its exponential is exactly $R(n,k;q)$, while
its real part is $\log|R(n,k;q)|$. Proposition~\ref{prop:complex-local} and
Corollary~\ref{cor:real-Lambda} give the complex and real local estimates
near the centre.

We use $q$ as the one-parameter variable in this section. For an integer
$n\ge0$ and $|q|<1$ define
\[
\pqbinom{n}{k}{q}
:=
\prod_{j=1}^{k}\frac{1-q^{n-k+j}}{1-q^j}
\qquad (0\le k\le n),
\]
where the empty product is one.
Since $|q|<1$, all factors $1-q^j$ with $j\ge1$ are non-zero. For
$n\ge2$ and $0\le k\le n-1$ the adjacent ratio is
\begin{equation}\label{eq:one-param-rho}
\frac{\pqbinom{n}{k+1}{q}}{\pqbinom{n}{k}{q}}
=\frac{1-q^{n-k}}{1-q^{k+1}}.
\end{equation}
Consequently, for $1\le k\le n-1$ the curvature ratio is
\begin{equation}\label{eq:R-basic}
R(n,k;q)
:=
\frac{\pqbinom{n}{k-1}{q}\,\pqbinom{n}{k+1}{q}}
     {\pqbinom{n}{k}{q}^{2}}
=
\frac{(1-q^{n-k})(1-q^k)}{(1-q^{n-k+1})(1-q^{k+1})}.
\end{equation}

The same ratio will be used in a hyperbolic form when $q=e^{-t}$. This
form is especially convenient on the real axis, where Taylor expansion of
$\log\sinh$ gives the constants $A$ and $B$.

\begin{lemma}\label{lem:R-hyperbolic}
Let $t\in\mathbb C$ with $\Real t>0$, let $q=e^{-t}$, and let
$1\le k\le n-1$. Then
\begin{equation}\label{eq:R-sinh}
R(n,k;e^{-t})
=
e^t\,
\frac{\sinh(\tfrac{k t}{2})\sinh(\tfrac{(n-k)t}{2})}
{\sinh(\tfrac{(k+1)t}{2})\sinh(\tfrac{(n-k+1)t}{2})}.
\end{equation}
\end{lemma}

\begin{proof}
Use $1-e^{-x}=2e^{-x/2}\sinh(x/2)$ in \eqref{eq:R-basic}. The
factors $2$ cancel. The exponential part is
\[
\frac{e^{-k t/2}e^{-(n-k)t/2}}
{e^{-(k+1)t/2}e^{-(n-k+1)t/2}}=e^t,
\]
and the remaining factors give \eqref{eq:R-sinh}.
\end{proof}

For complex $q$ it is safer not to choose a logarithm of the quotient
\eqref{eq:R-basic}. The following series is an ordinary absolutely
convergent series. It defines a logarithm of $R(n,k;q)$ only through its
exponential identity.

\begin{proposition}\label{prop:Lambda}
Let $t\in\mathbb C$ with $\Real t>0$, let $q=e^{-t}$, and let
$1\le k\le n-1$. Define
\begin{equation}\label{eq:Lambda-def}
\Lambda(n,k;q)
:=
-\sum_{m=1}^{\infty}\frac{1-q^m}{m}\left(q^{mk}+q^{m(n-k)}\right).
\end{equation}
Then the series converges absolutely and
\begin{equation}\label{eq:Lambda-exp}
e^{\Lambda(n,k;q)}=R(n,k;q).
\end{equation}
Moreover,
\begin{equation}\label{eq:Lambda-real}
\Real\Lambda(n,k;q)=\log|R(n,k;q)|.
\end{equation}
If
\[
\mu:=\min\{k,n-k\}\,\Real t,
\]
then $\mu>0$ and, for every integer $M\ge0$,
\begin{equation}\label{eq:Lambda-remainder}
\left|
\sum_{m=M+1}^{\infty}\frac{1-q^m}{m}\left(q^{mk}+q^{m(n-k)}\right)
\right|
\le
\frac{4e^{-(M+1)\mu}}{(M+1)(1-e^{-\mu})}.
\end{equation}
\end{proposition}

\begin{proof}
Since $|q|<1$ and $1\le k\le n-1$, one has
\[
|1-q^m|\le 1+|q|^m\le2,
\qquad
|q|^{mk}+|q|^{m(n-k)}\le2e^{-m\mu}.
\]
Hence
\[
\left|
\sum_{m=M+1}^{\infty}\frac{1-q^m}{m}\left(q^{mk}+q^{m(n-k)}\right)
\right|
\le
4\sum_{m=M+1}^{\infty}\frac{e^{-m\mu}}{m}.
\]
Since
\[
\sum_{m=M+1}^{\infty}\frac{e^{-m\mu}}{m}
\le
\frac{e^{-(M+1)\mu}}{M+1}\sum_{j=0}^{\infty}e^{-j\mu}
=
\frac{e^{-(M+1)\mu}}{(M+1)(1-e^{-\mu})},
\]
this proves \eqref{eq:Lambda-remainder}. In particular, the series
\eqref{eq:Lambda-def} converges absolutely.

For $j\ge1$, the principal logarithm satisfies
\[
\log(1-q^j)=-\sum_{m=1}^{\infty}\frac{q^{mj}}{m},
\]
because $|q^j|<1$ and hence $1-q^j$ lies in the right half-plane. Thus
\begin{align*}
\Lambda(n,k;q)
&=
\log(1-q^k)-\log(1-q^{k+1})
+\log(1-q^{n-k})-\log(1-q^{n-k+1}).
\end{align*}
Exponentiating this identity gives
\[
e^{\Lambda(n,k;q)}
=
\frac{(1-q^k)(1-q^{n-k})}{(1-q^{k+1})(1-q^{n-k+1})}
=
R(n,k;q),
\]
which proves \eqref{eq:Lambda-exp}. Taking absolute values in
\eqref{eq:Lambda-exp} gives
\[
e^{\Real\Lambda(n,k;q)}=|R(n,k;q)|,
\]
and therefore \eqref{eq:Lambda-real}.
\end{proof}

We now return to the quotient parameter $r$ and let it depend on $n$ as in
the introduction. In the central range both $k\Real t$ and
$(n-k)\Real t$ stay bounded away from zero, while $t$ itself is small.
This permits the expansion of $1-e^{-m t}$ inside \eqref{eq:Lambda-def}.

\begin{proposition}\label{prop:complex-local}
Fix an admissible set $\mathcal D$ and a constant $L>0$. For
\[
r=\exp(-t),
\qquad
t=\frac{v}{n}+i\,\frac{\delta}{n^{3/4}},
\qquad
(v,\delta)\in \mathcal D,
\]
and $1\le k\le n-1$, define
\begin{align}
S_1(n,k;r)
&:=
\frac{1}{e^{k t}-1}
+
\frac{1}{e^{(n-k)t}-1},\label{eq:S1-def}\\[1mm]
S_2(n,k;r)
&:=
\frac{e^{k t}}{(e^{k t}-1)^2}
+
\frac{e^{(n-k)t}}{(e^{(n-k)t}-1)^2}.\label{eq:S2-def}
\end{align}
Then, for all sufficiently large $n$ and all integers $k$ with
\[
|k-n/2|\le L n^{3/4},
\]
one has $1\le k\le n-1$ and
\begin{align}
\Lambda(n,k;r)
&=
-t S_1(n,k;r)
+O_{\mathcal D,L}(n^{-3/2}),\label{eq:complex-first-order}\\[1mm]
\Lambda(n,k;r)
&=
-t S_1(n,k;r)
+\frac{t^2}{2}\,S_2(n,k;r)
+O_{\mathcal D,L}(|t|^3).\label{eq:complex-second-order}
\end{align}
\end{proposition}

\begin{proof}
For all sufficiently large $n$, the condition $|k-n/2|\le L n^{3/4}$
implies
\[
k\ge \frac n4,
\qquad
n-k\ge \frac n4.
\]
Set
\[
a:=e^{-k t},
\qquad
b:=e^{-(n-k)t},
\qquad
\sigma:=e^{-v_{\min}/4}.
\]
Then $0<\sigma<1$ and
\[
|a|=e^{-k\Real t}\le \sigma,
\qquad
|b|=e^{-(n-k)\Real t}\le \sigma.
\]

For the first expansion, use the identity
\[
1-e^{-x}=x-x^2\int_0^1(1-s)e^{-sx}\ud s.
\]
If $\Real x\ge0$, this gives
\[
|1-e^{-x}-x|\le \frac{|x|^2}{2}.
\]
With $x=m t$ we may write
\[
1-e^{-m t}=m t+\varepsilon_m,
\qquad
|\varepsilon_m|\le \frac{m^2|t|^2}{2}.
\]
Substitution in \eqref{eq:Lambda-def} gives
\[
\Lambda(n,k;r)
=
-t\sum_{m=1}^{\infty}(a^m+b^m)
-
\sum_{m=1}^{\infty}\frac{\varepsilon_m}{m}(a^m+b^m).
\]
The remainder is bounded by
\[
\left|
\sum_{m=1}^{\infty}\frac{\varepsilon_m}{m}(a^m+b^m)
\right|
\le
\frac{|t|^2}{2}\sum_{m=1}^{\infty}m(|a|^m+|b|^m)
\le
|t|^2\sum_{m=1}^{\infty}m\sigma^m
=
\frac{\sigma}{(1-\sigma)^2}|t|^2.
\]
Moreover,
\[
\sum_{m=1}^{\infty}(a^m+b^m)
=
\frac{a}{1-a}+\frac{b}{1-b}
=
\frac{1}{e^{k t}-1}+\frac{1}{e^{(n-k)t}-1}
=
S_1(n,k;r).
\]
Since $|t|=O_{\mathcal D}(n^{-3/4})$, this proves
\eqref{eq:complex-first-order}.

For the second expansion, use
\[
1-e^{-x}
=
x-\frac{x^2}{2}
+x^3\int_0^1\frac{(1-s)^2}{2}e^{-sx}\ud s.
\]
For $\Real x\ge0$,
\[
\left|1-e^{-x}-x+\frac{x^2}{2}\right|
\le
\frac{|x|^3}{6}.
\]
Thus
\[
1-e^{-m t}
=
m t-\frac{m^2t^2}{2}+\eta_m,
\qquad
|\eta_m|\le \frac{m^3|t|^3}{6}.
\]
Substitution in \eqref{eq:Lambda-def} gives
\[
\Lambda(n,k;r)
=
-t\sum_{m=1}^{\infty}(a^m+b^m)
+\frac{t^2}{2}\sum_{m=1}^{\infty}m(a^m+b^m)
-
\sum_{m=1}^{\infty}\frac{\eta_m}{m}(a^m+b^m).
\]
The last sum is
\[
O_{\mathcal D,L}\!\left(
|t|^3\sum_{m=1}^{\infty}m^2\sigma^m
\right)
=O_{\mathcal D,L}(|t|^3).
\]
Furthermore,
\[
\sum_{m=1}^{\infty}m(a^m+b^m)
=
\frac{a}{(1-a)^2}+\frac{b}{(1-b)^2}
=
\frac{e^{k t}}{(e^{k t}-1)^2}
+
\frac{e^{(n-k)t}}{(e^{(n-k)t}-1)^2}
=
S_2(n,k;r).
\]
This proves \eqref{eq:complex-second-order}.
\end{proof}

On the real axis the hyperbolic formula gives a sharper expansion. This is
where the two functions $A$ and $B$ enter the paper.

\begin{corollary}\label{cor:real-Lambda}
Fix constants $0<z_0\le z_1<\infty$ and $L>0$. Let
$z\in[z_0,z_1]$ and let $q=e^{-z/n}$. Assume that $n$ is even and write
$\kappa:=n/2$. Then, uniformly for $z\in[z_0,z_1]$, all sufficiently
large even $n$, and all integers $\ell$ with $|\ell|\le L n^{3/4}$,
\begin{equation}\label{eq:real-logR}
\log R\!\left(n,\kappa+\ell;e^{-z/n}\right)
=
-\frac{A(z)}{n}
-\frac{B(z)}{n^3}\,\ell^2
+O_{z_0,z_1,L}(n^{-2}),
\end{equation}
where $A$ and $B$ are the functions defined in \eqref{eq:intro-AB}. Consequently,
\begin{equation}\label{eq:real-logR-centred}
\log R\!\left(n,\kappa+\ell;e^{-z/n}\right)-\log R\!\left(n,\kappa;e^{-z/n}\right)
=
-\frac{B(z)}{n^3}\,\ell^2
+O_{z_0,z_1,L}(n^{-2}).
\end{equation}
\end{corollary}

\begin{proof}
Set
\[
a:=\frac z4,
\qquad
h:=\frac{z}{2n},
\qquad
\Delta:=\frac{z\ell}{2n}.
\]
Since $z\in[z_0,z_1]$ and $|\ell|\le L n^{3/4}$, one has
$\Delta=O_{z_0,z_1,L}(n^{-1/4})$ and $h=O_{z_1}(n^{-1})$. For all
sufficiently large $n$, the points $a\pm\Delta$ and
$a+h\pm\Delta$ lie in a compact subinterval of $(0,\infty)$ that
only depends on $z_0,z_1$, and $L$.

By Lemma~\ref{lem:R-hyperbolic},
\[
\log R\!\left(n,\kappa+\ell;e^{-z/n}\right)
=
\frac zn
+\psi(a+\Delta)+\psi(a-\Delta)
-\psi(a+h+\Delta)-\psi(a+h-\Delta),
\]
where $\psi(y):=\log\sinh y$. Taylor's formula in the variable $h$
gives, uniformly for the above values of $y$,
\[
\psi(y)-\psi(y+h)
=
-h\psi'(y)+O_{z_0,z_1,L}(n^{-2}).
\]
Therefore
\[
\log R
=
\frac zn
-h\left(\psi'(a+\Delta)+\psi'(a-\Delta)\right)
+O_{z_0,z_1,L}(n^{-2}).
\]
Taylor's formula in the variable $\Delta$ gives
\[
\psi'(a+\Delta)+\psi'(a-\Delta)
=
2\psi'(a)+\psi^{(3)}(a)\Delta^2+O_{z_0,z_1,L}(\Delta^4).
\]
Since $h\Delta^4=O_{z_0,z_1,L}(n^{-2})$, substitution yields
\[
\log R
=
\frac zn-2h\psi'(a)-h\psi^{(3)}(a)\Delta^2
+O_{z_0,z_1,L}(n^{-2}).
\]
Now $\psi'(a)=\coth a$, and hence
\[
\frac zn-2h\psi'(a)
=
-\frac{z}{n}\left(\coth\frac z4-1\right)
=
-\frac{A(z)}{n}.
\]
Moreover,
\[
\psi^{(3)}(a)=2\coth(a)\csch^2(a)=\frac{\sinh(2a)}{\sinh^4(a)},
\]
and therefore
\[
-h\psi^{(3)}(a)\Delta^2
=
-\frac{z}{2n}\,\psi^{(3)}\!\left(\frac z4\right)
\frac{z^2\ell^2}{4n^2}
=
-\frac{B(z)}{n^3}\,\ell^2.
\]
This proves \eqref{eq:real-logR}. The same estimate with $\ell=0$ is
valid uniformly; subtracting it from \eqref{eq:real-logR} gives
\eqref{eq:real-logR-centred}.
\end{proof}

\section{Reduction from $(p,q)$ to one parameter}\label{sec:reduction}

The following proposition collects the elementary identities used to pass
from the one-parameter estimates to the full $(p,q)$-coefficients.

\begin{proposition}\label{prop:pq-reduction}
Let $n\ge1$ be an integer, and let $p,q\in\mathbb C$ satisfy $p\ne0$ and
$p^j\ne q^j$ for $1\le j\le n$. Set $r=q/p$. Then, for $0\le k\le n$,
\begin{equation}\label{eq:pq-factor}
\pqbinom{n}{k}{p,q}=p^{k(n-k)}\pqbinom{n}{k}{r},
\end{equation}
and
\begin{equation}\label{eq:pq-coefficient-symmetry}
\pqbinom{n}{k}{p,q}=\pqbinom{n}{n-k}{p,q}.
\end{equation}
The adjacent and second-order ratios satisfy
\begin{align}
\rho_k(p,q)&=p^{n-2k-1}\rho_k(r),
\qquad 0\le k\le n-1,\label{eq:rho-red}\\
R(n,k;p,q)&=p^{-2}R(n,k;r),
\qquad 1\le k\le n-1.\label{eq:R-red}
\end{align}
Consequently, for $g$ in \eqref{eq:intro-g},
\begin{equation}\label{eq:g-second-diff}
\Delta^2g(k)=\log|R(n,k;r)|-2\log|p|,
\qquad 1\le k\le n-1.
\end{equation}
\end{proposition}

\begin{proof}
Factoring $p^{n-k}$ from each factor in \eqref{eq:intro-pq-def} gives
\eqref{eq:pq-factor}. All denominators are non-zero by the assumptions.
Moreover,
\[
\pqbinom{n}{k}{r}=\frac{(r;r)_n}{(r;r)_k(r;r)_{n-k}},
\qquad
(r;r)_m:=\prod_{j=1}^{m}(1-r^j),
\]
where $(r;r)_0=1$. This expression and the factor $p^{k(n-k)}$ are
invariant under $k\mapsto n-k$, which proves
\eqref{eq:pq-coefficient-symmetry}. Taking adjacent ratios in
\eqref{eq:pq-factor} and using
\[
(k+1)(n-k-1)-k(n-k)=n-2k-1
\]
gives \eqref{eq:rho-red}. Dividing this identity by the one with $k-1$
in place of $k$ gives \eqref{eq:R-red}. Taking logarithms of moduli
gives \eqref{eq:g-second-diff}.
\end{proof}

\section{Complex local estimates near the centre}\label{sec:complex-central}

This section treats the complex case near $k=n/2$. We use the
$n$-dependent parameters from \eqref{eq:intro-parameters}--\eqref{eq:intro-r-exponent}.
Thus $r=e^{-t}$ with $\Real t=v/n>0$. In estimates involving $g$
or the full coefficients we assume $|p|=e^{-u/n}$. The argument of $p$ does
not enter these absolute-value estimates.

Propositions~\ref{prop:Lambda}, \ref{prop:complex-local}, and
\ref{prop:pq-reduction} give the following direct consequences for the
coefficient moduli. The estimates for centred ratios are then proved in
Theorems~\ref{thm:complex-middle} and \ref{thm:complex-sqrt}.

\begin{corollary}\label{cor:complex-curvature}
Fix an admissible set $\mathcal D$, a constant $L>0$, and a real number $u$.
Let
\[
t=\frac vn+i\frac{\delta}{n^{3/4}},
\qquad (v,\delta)\in\mathcal D,
\qquad r=e^{-t},
\qquad |p|=e^{-u/n},
\qquad q=pr.
\]
For all sufficiently large $n$, uniformly for $(v,\delta)\in\mathcal D$
and all integers $k$ with $|k-n/2|\le Ln^{3/4}$,
\begin{align}
\log|R(n,k;r)|
&=-\Real\bigl(tS_1(n,k;r)\bigr)
+O_{\mathcal D,L}(n^{-3/2}),\label{eq:complex-curvature-R}\\
\Delta^2g(k)
&=\frac{2u}{n}-\Real\bigl(tS_1(n,k;r)\bigr)
+O_{\mathcal D,L}(n^{-3/2}).\label{eq:complex-curvature-g}
\end{align}
In particular,
\[
\log|R(n,k;r)|=O_{\mathcal D,L}(n^{-3/4}),
\qquad
\Delta^2g(k)=O_{\mathcal D,L,u}(n^{-3/4}).
\]
If $n$ is even and $\kappa=n/2$, then
\begin{equation}\label{eq:complex-center}
\Delta^2g(\kappa)
=\frac{2u}{n}-2\Real\left(\frac{t}{e^{\kappa t}-1}\right)
+O_{\mathcal D}(n^{-3/2}).
\end{equation}
For $\delta=0$, this improves to
\begin{equation}\label{eq:complex-center-real}
\Delta^2g(\kappa)=\frac{2u-A(v)}{n}+O_{\mathcal D}(n^{-2}).
\end{equation}
\end{corollary}

\begin{proof}
The central range lies in $1\le k\le n-1$ for all sufficiently large $n$.
Take real parts in \eqref{eq:complex-first-order} and use
\eqref{eq:Lambda-real} and \eqref{eq:g-second-diff}. This gives the first
two estimates. Since $k,n-k\ge n/4$ in this range, the denominators in
$S_1$ have moduli at least $e^{v_{\min}/4}-1$. Thus $S_1$ is uniformly
bounded, and the stated bounds follow from
$|t|=O_{\mathcal D}(n^{-3/4})$.
At $k=\kappa$, the identity
\[
S_1(n,\kappa;r)=\frac{2}{e^{\kappa t}-1}
\]
gives \eqref{eq:complex-center}. If $\delta=0$, apply
Corollary~\ref{cor:real-Lambda} at the centre and use
$-2\log|p|=2u/n$ to obtain \eqref{eq:complex-center-real}.
\end{proof}

Curvature gives second-order differences. For centred ratios it is more direct to
multiply adjacent ratios and then expand the resulting product. The next
lemma gives an exact logarithmic form of that product.

\begin{lemma}\label{lem:complex-middle-exact}
Assume that $n$ is even and write $\kappa:=n/2$. Let
$r=e^{-t}$ with $\Real t>0$. Let $0\le \ell\le \kappa$ and set
\begin{equation}\label{eq:Epm-def}
E_-(n,\ell;r):=e^{-(\kappa-\ell+1)t},
\qquad
E_0(n;r):=e^{-(\kappa+1)t},
\qquad
E_+(n,\ell;r):=e^{-(\kappa+\ell+1)t}.
\end{equation}
Then
\begin{equation}\label{eq:complex-middle-exact}
\log\left|\frac{\pqbinom{n}{\kappa+\ell}{r}}{\pqbinom{n}{\kappa}{r}}\right|
=
-\Real\sum_{m=1}^{\infty}
\frac{E_-(n,\ell;r)^m-2E_0(n;r)^m+E_+(n,\ell;r)^m}
{m\,(1-e^{-m t})}.
\end{equation}
The series converges absolutely.
\end{lemma}

\begin{proof}
Since $\Real t>0$, one has $|r|<1$. The case $\ell=0$ gives zero on both
sides, so assume first that $1\le \ell\le\kappa$. From \eqref{eq:intro-rho},
\[
\rho_k(r)=\frac{1-r^{n-k}}{1-r^{k+1}},
\]
and therefore
\[
\frac{\pqbinom{n}{\kappa+\ell}{r}}{\pqbinom{n}{\kappa}{r}}
=
\prod_{s=0}^{\ell-1}\rho_{\kappa+s}(r)
=
\prod_{j=1}^{\ell}\frac{1-r^{\kappa-\ell+j}}{1-r^{\kappa+j}}.
\]
For $|w|<1$,
\[
\log|1-w|=-\Real\sum_{m=1}^{\infty}\frac{w^m}{m}.
\]
Applying this identity to each factor gives
\[
\log\left|\frac{\pqbinom{n}{\kappa+\ell}{r}}{\pqbinom{n}{\kappa}{r}}\right|
=
-\Real\sum_{m=1}^{\infty}\frac{1}{m}
\sum_{j=1}^{\ell}\left(r^{m(\kappa-\ell+j)}-r^{m(\kappa+j)}\right).
\]
The inner sum is
\[
\sum_{j=1}^{\ell}r^{m(\kappa-\ell+j)}
-
\sum_{j=1}^{\ell}r^{m(\kappa+j)}
=
\frac{r^{m(\kappa-\ell+1)}(1-r^{m\ell})^2}{1-r^m}
=
\frac{E_-^m-2E_0^m+E_+^m}{1-e^{-m t}},
\]
which proves \eqref{eq:complex-middle-exact}.

For absolute convergence, set
\[
\sigma:=|r|^{\kappa-\ell+1}<1.
\]
Then
\[
|E_-(n,\ell;r)|,\ |E_0(n;r)|,\ |E_+(n,\ell;r)|\le \sigma,
\qquad
|1-r^m|\ge 1-|r|^m\ge 1-|r|.
\]
Therefore
\[
\left|\frac{E_-^m-2E_0^m+E_+^m}{m(1-e^{-m t})}\right|
\le
\frac{4\sigma^m}{m(1-|r|)},
\]
and $\sum_{m\ge1}\sigma^m/m$ converges. This proves absolute convergence.
\end{proof}

We next keep one further term in the expansion of
$(1-e^{-m t})^{-1}$. This adds the series $\Phi_0$ to the
$\Phi_1$- and $\Phi_2$-terms.

\begin{theorem}\label{thm:complex-middle}
Fix an admissible set $\mathcal D$, a constant $L>0$, and a real number $u$.
Assume that $n$ is even and write $\kappa:=n/2$. Let
$r=e^{-t}$ with $t$ as in \eqref{eq:intro-r-exponent} and
$(v,\delta)\in \mathcal D$, and let $p$ satisfy $|p|=e^{-u/n}$ in the definition
of $g$ in \eqref{eq:intro-g}. Then, uniformly for all $(v,\delta)\in \mathcal D$,
all sufficiently large even $n$, and every integer $\ell$ with
$|\ell|\le L n^{3/4}$,
\begin{align}
\log\left|\frac{\pqbinom{n}{\kappa+\ell}{r}}{\pqbinom{n}{\kappa}{r}}\right|
&=
-\Real\left(
\frac{\Phi_2(n,|\ell|;r)}{t}
+\frac{1}{2}\,\Phi_1(n,|\ell|;r)
+\frac{t}{12}\,\Phi_0(n,|\ell|;r)
\right)
\nonumber\\
&\quad
+O_{\mathcal D,L}(n^{-9/4}),\label{eq:complex-middle-main}\\[1mm]
g(\kappa+\ell)-g(\kappa)
&=
\frac{u}{n}\,\ell^2
-\Real\left(
\frac{\Phi_2(n,|\ell|;r)}{t}
+\frac{1}{2}\,\Phi_1(n,|\ell|;r)
+\frac{t}{12}\,\Phi_0(n,|\ell|;r)
\right)
\nonumber\\
&\quad
+O_{\mathcal D,L}(n^{-9/4}),\label{eq:complex-middle-g}
\end{align}
where
\begin{align}
\Phi_0(n,\ell;r)
&:=
\sum_{m=1}^{\infty}
\left(E_-(n,\ell;r)^m-2E_0(n;r)^m+E_+(n,\ell;r)^m\right),\label{eq:Phi0-def}\\[1mm]
\Phi_1(n,\ell;r)
&:=
\sum_{m=1}^{\infty}\frac{E_-(n,\ell;r)^m-2E_0(n;r)^m+E_+(n,\ell;r)^m}{m},\label{eq:Phi1-def}\\[1mm]
\Phi_2(n,\ell;r)
&:=
\sum_{m=1}^{\infty}\frac{E_-(n,\ell;r)^m-2E_0(n;r)^m+E_+(n,\ell;r)^m}{m^2}.\label{eq:Phi2-def}
\end{align}
The three series converge absolutely, and
\begin{align}
\Phi_0(n,\ell;r)
&=
\frac{E_-(n,\ell;r)}{1-E_-(n,\ell;r)}
-2\,\frac{E_0(n;r)}{1-E_0(n;r)}
+\frac{E_+(n,\ell;r)}{1-E_+(n,\ell;r)},\label{eq:Phi0-rational}\\[1mm]
\Phi_1(n,\ell;r)
&=
-\log\left(1-E_-(n,\ell;r)\right)
+2\log\left(1-E_0(n;r)\right)
-\log\left(1-E_+(n,\ell;r)\right),\label{eq:Phi1-log}\\[1mm]
\Phi_2(n,\ell;r)
&=
\Li_2\left(E_-(n,\ell;r)\right)
-2\Li_2\left(E_0(n;r)\right)
+\Li_2\left(E_+(n,\ell;r)\right).\label{eq:Phi2-dilog}
\end{align}
Here the logarithms in \eqref{eq:Phi1-log} are principal logarithms.
\end{theorem}

\begin{proof}
The one-parameter coefficients satisfy
$\pqbinom{n}{\kappa+\ell}{r}=\pqbinom{n}{\kappa-\ell}{r}$, and hence
$g(\kappa+\ell)=g(\kappa-\ell)$. It is therefore enough to consider
$0\le \ell\le L n^{3/4}$. For all sufficiently large $n$ this range is
contained in $0\le\ell\le\kappa$.
Let
\[
\sigma:=e^{-v_{\min}/4}<1.
\]
For all sufficiently large $n$ and all such $\ell$ one has
$\kappa-\ell+1\ge n/4$, hence
\[
|E_-(n,\ell;r)|,\ |E_0(n;r)|,\ |E_+(n,\ell;r)|\le \sigma.
\]
In particular, the series defining $\Phi_0$, $\Phi_1$, and $\Phi_2$ converge
absolutely.

Choose
\[
M:=\lfloor n^{1/8}\rfloor.
\]
Since $|t|=O_{\mathcal D}(n^{-3/4})$, one has
$M|t|=O_{\mathcal D}(n^{-5/8})$. The function
\[
x\longmapsto
\frac{1}{1-e^{-x}}-\frac1x-\frac12-\frac{x}{12}
\]
is analytic at $x=0$. Therefore, uniformly for $1\le m\le M$,
\begin{equation}\label{eq:inv-den-expansion}
\frac{1}{1-e^{-m t}}
=
\frac{1}{m t}+\frac12+\frac{m t}{12}
+O_{\mathcal D}(m^3|t|^3).
\end{equation}

Set
\[
c_m:=E_-(n,\ell;r)^m-2E_0(n;r)^m+E_+(n,\ell;r)^m.
\]
By Lemma~\ref{lem:complex-middle-exact},
\[
\log\left|\frac{\pqbinom{n}{\kappa+\ell}{r}}{\pqbinom{n}{\kappa}{r}}\right|
=
-\Real\sum_{m=1}^{\infty}\frac{c_m}{m(1-e^{-m t})}.
\]
For $m\le M$, insert \eqref{eq:inv-den-expansion}. Since $|c_m|\le4\sigma^m$,
\[
\sum_{m=1}^{M}
\left|\frac{c_m}{m}\,O_{\mathcal D}(m^3|t|^3)\right|
\le
4|t|^3\sum_{m=1}^{\infty}m^2\sigma^m
=
O_{\mathcal D,L}(n^{-9/4}).
\]
Hence
\[
\sum_{m=1}^{M}\frac{c_m}{m(1-e^{-m t})}
=
\frac{1}{t}\sum_{m=1}^{M}\frac{c_m}{m^2}
+
\frac12\sum_{m=1}^{M}\frac{c_m}{m}
+
\frac{t}{12}\sum_{m=1}^{M}c_m
+O_{\mathcal D,L}(n^{-9/4}).
\]

It remains to estimate the remaining part. For $m\ge1$,
\[
|1-e^{-m t}|
\ge 1-e^{-m v/n}
\ge c_{\mathcal D}\min\left(\frac{m}{n},1\right),
\]
where $c_{\mathcal D}>0$ depends only on $v_{\min}$. Therefore
\[
\sum_{m=M+1}^{\infty}
\left|\frac{c_m}{m(1-e^{-m t})}\right|
\le
C_{\mathcal D}\sum_{m=M+1}^{\infty}
\frac{\sigma^m}{m\min(m/n,1)}
=
O_{\mathcal D}\!\left(e^{-c_{\mathcal D} n^{1/8}}\right).
\]
Also,
\[
\sum_{m=M+1}^{\infty}|c_m|
+
\sum_{m=M+1}^{\infty}\left|\frac{c_m}{m}\right|
\le
C_{\mathcal D}\sum_{m=M+1}^{\infty}\sigma^m
=
O_{\mathcal D}\!\left(e^{-c_{\mathcal D} n^{1/8}}\right),
\]
and, since $|t|\ge v_{\min}/n$,
\[
\frac{1}{|t|}
\sum_{m=M+1}^{\infty}\left|\frac{c_m}{m^2}\right|
\le
\frac{4}{|t|}\sum_{m=M+1}^{\infty}\frac{\sigma^m}{m^2}
=
O_{\mathcal D}\!\left(e^{-c_{\mathcal D} n^{1/8}}\right).
\]
The exponential bounds absorb the polynomial factors in $n$. Combining the
estimates proves \eqref{eq:complex-middle-main}. The formula
\eqref{eq:complex-middle-g} follows from
\[
g(\kappa+\ell)-g(\kappa)
=
\log\left|\frac{\pqbinom{n}{\kappa+\ell}{r}}{\pqbinom{n}{\kappa}{r}}\right|
+\frac{u}{n}\,\ell^2,
\]
which is the contribution of the factor $p^{k(n-k)}$.
Finally, \eqref{eq:Phi0-rational}, \eqref{eq:Phi1-log}, and
\eqref{eq:Phi2-dilog} follow from the absolutely convergent power series for
$w/(1-w)$, $-\log(1-w)$, and $\Li_2(w)$ on $|w|<1$.
\end{proof}

The following symmetric Green identity converts second-order differences to
centred ratios. It is used in Theorem~\ref{thm:complex-sqrt} and again in
Section~\ref{sec:small-imaginary-perturbation}.

\begin{lemma}\label{lem:green}
Let $f:\mathbb Z\to\mathbb C$ and define
\[
(\Delta^2 f)(k):=f(k+1)-2f(k)+f(k-1).
\]
Fix an integer $\kappa$.
Then for every integer $m\ge1$,
\begin{equation}\label{eq:green-symmetric}
f(\kappa+m)+f(\kappa-m)-2f(\kappa)
=
\sum_{j=-(m-1)}^{m-1}(m-|j|)\,\Delta^2 f(\kappa+j).
\end{equation}
\end{lemma}

\begin{proof}
Set $\Delta f(k):=f(k+1)-f(k)$.
Then
\[
f(\kappa+m)-f(\kappa)=\sum_{u=0}^{m-1}\Delta f(\kappa+u),
\qquad
f(\kappa-m)-f(\kappa)=-\sum_{u=1}^{m}\Delta f(\kappa-u).
\]
Adding these two identities gives
\[
f(\kappa+m)+f(\kappa-m)-2f(\kappa)
=
\sum_{u=0}^{m-1}\left(\Delta f(\kappa+u)-\Delta f(\kappa-u-1)\right).
\]
For each $u\ge0$,
\[
\Delta f(\kappa+u)-\Delta f(\kappa-u-1)
=
\sum_{j=-u}^{u}\Delta^2 f(\kappa+j)
\]
by telescoping.
Summing over $u=0,\dots,m-1$ and exchanging the order of summation, the coefficient of $\Delta^2 f(\kappa+j)$ is the number of integers $u\in\{0,\dots,m-1\}$ with $|j|\le u$, namely $m-|j|$.
This proves \eqref{eq:green-symmetric}.
\end{proof}

We also give a complex estimate for deviations of order $\sqrt n$. It follows
from the second expansion in Proposition~\ref{prop:complex-local} and the
same Green identity.

\begin{theorem}\label{thm:complex-sqrt}
Let $z>0$, $\alpha\in\mathbb R$, and
\[
r=\exp(-t),
\qquad
t:=\frac zn+i\frac{\alpha}{n^{3/4}}.
\]
Assume that $n$ is even and write $\kappa:=n/2$. Put
\[
\tau:=z+i\alpha n^{1/4},
\qquad
w:=\frac{\tau}{2},
\]
and define
\[
F(w):=\frac{1}{e^w-1},
\qquad
G(w):=\frac{e^w}{(e^w-1)^2},
\qquad
H(w):=\frac{e^w(e^w+1)}{(e^w-1)^3}.
\]
For each fixed $L>0$, uniformly for all integers $\ell$ with
$|\ell|\le L\sqrt n$,
\begin{align}
\log\left|\frac{\pqbinom{n}{\kappa+\ell}{r}}{\pqbinom{n}{\kappa}{r}}\right|
&=
-\ell^2\Real\left(tF(w)\right)
+\frac{\ell^2}{2}\Real\left(t^2G(w)\right)
-\frac{\ell^2(\ell^2-1)}{12}\Real\left(t^3H(w)\right)
\nonumber\\
&\quad
+O_{z,\alpha,L}\!\left(n^3|t|^5+n^2|t|^4+n|t|^3\right).
\label{eq:complex-sqrt}
\end{align}
Equivalently, writing $x=\ell/\sqrt n$, one has
\begin{align}
\log\left|\frac{\pqbinom{n}{\kappa+\ell}{r}}{\pqbinom{n}{\kappa}{r}}\right|
&=
-x^2 C_n(z,\alpha)
+\frac{x^2}{2n}\Real\left(\tau^2G(w)\right)
-\frac{x^4}{12n}\Real\left(\tau^3H(w)\right)
\nonumber\\
&\quad
+O_{z,\alpha,L}\!\left(n^3|t|^5+n^2|t|^4+n|t|^3\right),
\label{eq:complex-sqrt-x}
\end{align}
where
\begin{equation}\label{eq:Cn-def}
C_n(z,\alpha):=
\Real\left(\tau F(w)\right)
=
\Real\left(
\frac{z+i\alpha n^{1/4}}{e^{z/2+i\alpha n^{1/4}/2}-1}
\right).
\end{equation}
\end{theorem}

\begin{proof}
The quantity on the left of \eqref{eq:complex-sqrt} is even in $\ell$, and
the right-hand side has the same symmetry. It is therefore enough to consider
$0\le \ell\le L\sqrt n$. The case $\ell=0$ is immediate.
Set
\[
f_n(k):=\log \left|\pqbinom{n}{k}{r}\right|.
\]
Then
\[
\Delta^2 f_n(k)=\log|R(n,k;r)|.
\]
For $|j|\le L\sqrt n$ and all sufficiently large $n$,
\[
\frac z4\le \Real(w\pm jt)\le \frac{3z}{4}.
\]
On this vertical strip the required derivatives of $F$ and $G$ are bounded,
with constants depending only on $z$, $\alpha$, and $L$. Since $F''=H$,
Taylor's formula gives
\[
F(w+jt)+F(w-jt)
=
2F(w)+j^2t^2H(w)+O_{z,\alpha,L}(j^4|t|^4)
\]
and
\[
G(w+jt)+G(w-jt)
=
2G(w)+O_{z,\alpha,L}(j^2|t|^2).
\]
The second expansion in Proposition~\ref{prop:complex-local}, applied to
the fixed pair $(z,\alpha)$, gives
\begin{align*}
\Delta^2 f_n(\kappa+j)
&=
-\Real\left[
t\left(F(w+jt)+F(w-jt)\right)
\right]\\
&\quad+
\frac12\Real\left[
t^2\left(G(w+jt)+G(w-jt)\right)
\right]
+O_{z,\alpha,L}(|t|^3).
\end{align*}
Hence
\begin{align*}
\Delta^2 f_n(\kappa+j)
&=
-2\Real\left(tF(w)\right)
+\Real\left(t^2G(w)\right)
-j^2\Real\left(t^3H(w)\right)\\
&\quad
+O_{z,\alpha,L}\!\left(j^4|t|^5+j^2|t|^4+|t|^3\right)
\end{align*}
uniformly for $|j|\le L\sqrt n$.

Lemma~\ref{lem:green}, applied to $f_n$, gives
\[
2\left(f_n(\kappa+\ell)-f_n(\kappa)\right)
=
\sum_{j=-(\ell-1)}^{\ell-1}(\ell-|j|)\Delta^2 f_n(\kappa+j).
\]
Using
\[
\sum_{j=-(\ell-1)}^{\ell-1}(\ell-|j|)=\ell^2,
\qquad
\sum_{j=-(\ell-1)}^{\ell-1}(\ell-|j|)j^2
=
\frac{\ell^2(\ell^2-1)}6,
\]
and
\[
\sum_{j=-(\ell-1)}^{\ell-1}(\ell-|j|)j^4=O(\ell^6),
\]
we obtain
\begin{align*}
f_n(\kappa+\ell)-f_n(\kappa)
&=
-\ell^2\Real\left(tF(w)\right)
+\frac{\ell^2}{2}\Real\left(t^2G(w)\right)\\
&\quad
-\frac{\ell^2(\ell^2-1)}{12}\Real\left(t^3H(w)\right)\\
&\quad
+O_{z,\alpha,L}\!\left(\ell^6|t|^5+\ell^4|t|^4+\ell^2|t|^3\right).
\end{align*}
Since $\ell=O_L(\sqrt n)$, this proves \eqref{eq:complex-sqrt}. Replacing
$t$ by $\tau/n$ and $\ell$ by $x\sqrt n$ gives
\eqref{eq:complex-sqrt-x}; the difference between
$\ell^2(\ell^2-1)t^3/12$ and $x^4\tau^3/(12n)$ is absorbed by the
error term.
\end{proof}

\begin{remark}
The preceding theorem gives a direct comparison with the real critical value
for the full $(p,q)$-coefficient. Assume that $|p|=e^{-u/n}$ and let
$g(k)=\log\left|\pqbinom{n}{k}{p,q}\right|$. If $|x|\le L$ with
$x=\ell/\sqrt n$ and fixed $L$, then \eqref{eq:complex-sqrt-x} and the
factor $p^{k(n-k)}$ give
\begin{align}
g(\kappa+\ell)-g(\kappa)
&=
x^2\left(u-C_n(z,\alpha)\right)
+\frac{x^2}{2n}\Real\left(\tau^2G(w)\right)
-\frac{x^4}{12n}\Real\left(\tau^3H(w)\right)
\nonumber\\
&\quad
+O_{z,\alpha,L}\!\left(n^3|t|^5+n^2|t|^4+n|t|^3\right).
\label{eq:complex-sqrt-full-modulus}
\end{align}
For $\alpha=0$,
\[
C_n(z,0)=\frac{z}{e^{z/2}-1}=\frac{A(z)}2.
\]
Thus the choice $u=A(z)/2$ makes the first term in
\eqref{eq:complex-sqrt-full-modulus} vanish when $r$ is real. If
$\alpha\ne0$, write $b=\alpha n^{1/4}/2$. Then
\begin{equation}\label{eq:Cn-explicit}
C_n(z,\alpha)=
\frac{z(e^{z/2}\cos b-1)+\alpha n^{1/4}e^{z/2}\sin b}
{e^z-2e^{z/2}\cos b+1}.
\end{equation}
Hence the same choice $u=A(z)/2$ gives the coefficient
$A(z)/2-C_n(z,\alpha)$, which depends on $n$ when $\alpha\ne0$.
\end{remark}

\section{An imaginary perturbation of order $n^{-5/4}$ at the real critical value}\label{sec:small-imaginary-perturbation}

The main complex theorem treats $\Imag t$ of order $n^{-3/4}$ in $r=q/p$.
For comparison with the positive real critical value, we also treat an
imaginary perturbation of order $n^{-5/4}$.  For $|\ell|=O(n^{3/4})$ this perturbation changes
the quadratic coefficient but leaves the quartic coefficient unchanged.

\begin{proposition}\label{prop:small-imaginary-perturbation}
Let $z>0$, let $\alpha,\gamma\in\mathbb R$, and assume that $n$ is even. Put
$\kappa:=n/2$ and
\[
t:=\frac zn+i\frac{\alpha}{n^{5/4}},
\qquad
r:=e^{-t},
\qquad
u:=\frac{A(z)}2+\frac{\gamma}{\sqrt n},
\]
\[
p:=e^{-u/n},
\qquad
q:=p r.
\]
For every fixed $L>0$, uniformly for all integer sequences $\ell_n$ such that
$x_n:=\ell_n/n^{3/4}$ satisfies $|x_n|\le L$,
\begin{equation}\label{eq:small-imaginary-perturbation}
\left|
\frac{\pqbinom{n}{\kappa+\ell_n}{p,q}}
{\pqbinom{n}{\kappa}{p,q}}
\right|
=
\exp\!\left(
\left(\gamma+\frac{\alpha^2}{4}A''(z)\right)x_n^2
-\frac{B(z)}{12}x_n^4
+O_{z,\alpha,\gamma,L}(n^{-1/2})\right).
\end{equation}
The same formula holds with $\kappa-\ell_n$ in place of $\kappa+\ell_n$.
In particular, if $x_n\to x$, then
\[
\left|
\frac{\pqbinom{n}{\kappa\pm \ell_n}{p,q}}
{\pqbinom{n}{\kappa}{p,q}}
\right|
\longrightarrow
\exp\!\left(
\left(\gamma+\frac{\alpha^2}{4}A''(z)\right)x^2
-\frac{B(z)}{12}x^4
\right).
\]
For $\gamma=0$, the imaginary perturbation contributes the quadratic term
$\alpha^2A''(z)x^2/4$ at the real critical value $2u=A(z)$.
\end{proposition}

\begin{proof}
Let
\[
\tau:=nt=z+i\alpha n^{-1/4}.
\]
For $\tau$ in a neighbourhood of $z$ define
\[
\mathcal A(\tau):=\tau\left(\coth\frac{\tau}{4}-1\right),
\qquad
\mathcal B(\tau):=\frac{\tau^3\sinh(\tau/2)}{8\sinh^4(\tau/4)}.
\]
We first prove the local curvature estimate
\begin{equation}\label{eq:small-complex-logR}
\log\left|R(n,\kappa+j;e^{-t})\right|
=
-\frac{\Real\mathcal A(\tau)}{n}
-\frac{\Real\mathcal B(\tau)}{n^3}j^2
+O_{z,\alpha,L}(n^{-2})
\end{equation}
uniformly for $|j|\le L n^{3/4}$.
For all sufficiently large $n$, the points used below lie in a compact subset
of the domain on which one branch of $\log\sinh$ is analytic. Put
\[
a:=\frac{\tau}{4},
\qquad
h:=\frac{\tau}{2n},
\qquad
\Delta:=\frac{\tau j}{2n},
\]
and write $\psi(y):=\log\sinh y$ on this compact set. Lemma~\ref{lem:R-hyperbolic}
gives
\[
\log\left|R(n,\kappa+j;e^{-t})\right|
=
\Real\left(
\frac{\tau}{n}
+\psi(a+\Delta)+\psi(a-\Delta)
-\psi(a+h+\Delta)-\psi(a+h-\Delta)
\right).
\]
Taylor's formula in $h$ and then in $\Delta$ gives, uniformly in the
stated range of $j$,
\[
\log|R|
=
\Real\left(
\frac{\tau}{n}
-h\left(\psi'(a+\Delta)+\psi'(a-\Delta)\right)
\right)
+O_{z,\alpha,L}(n^{-2})
\]
and
\[
\psi'(a+\Delta)+\psi'(a-\Delta)
=
2\psi'(a)+\psi^{(3)}(a)\Delta^2+O_{z,\alpha,L}(\Delta^4).
\]
Since $h\Delta^4=O_{z,\alpha,L}(n^{-2})$, and since
$\psi'(a)=\coth a$ and
\[
\psi^{(3)}(a)=\frac{\sinh(2a)}{\sinh^4 a},
\]
this proves \eqref{eq:small-complex-logR}.

Let
\[
g_n(k):=\log\left|\pqbinom{n}{k}{p,q}\right|.
\]
By Proposition~\ref{prop:pq-reduction},
\[
\Delta^2g_n(\kappa+j)
=
\log\left|R(n,\kappa+j;e^{-t})\right|+\frac{2u}{n}.
\]
Together with \eqref{eq:small-complex-logR}, this gives
\[
\Delta^2g_n(\kappa+j)
=
\frac{2u-\Real\mathcal A(\tau)}{n}
-\frac{\Real\mathcal B(\tau)}{n^3}j^2
+O_{z,\alpha,L}(n^{-2}).
\]
Applying Lemma~\ref{lem:green} gives, uniformly for $|\ell|\le L n^{3/4}$,
\begin{align*}
g_n(\kappa+\ell)-g_n(\kappa)
&=
\frac{2u-\Real\mathcal A(\tau)}{2n}\ell^2
-\frac{\Real\mathcal B(\tau)}{12n^3}\ell^2(\ell^2-1)\\
&\quad
+O_{z,\alpha,L}\left(\frac{\ell^2}{n^2}\right).
\end{align*}
Now $\mathcal A$ and $\mathcal B$ are analytic near $z$ and real on the real
axis. Hence
\[
\Real\mathcal A(z+i\alpha n^{-1/4})
=
A(z)-\frac{\alpha^2}{2\sqrt n}A''(z)+O_{z,\alpha}(n^{-1}),
\]
and
\[
\Real\mathcal B(z+i\alpha n^{-1/4})=B(z)+O_{z,\alpha}(n^{-1/2}).
\]
Since $u=A(z)/2+\gamma n^{-1/2}$ and $\ell=x_n n^{3/4}$, we obtain
\[
g_n(\kappa+\ell)-g_n(\kappa)
=
\left(\gamma+\frac{\alpha^2}{4}A''(z)\right)x_n^2
-\frac{B(z)}{12}x_n^4
+O_{z,\alpha,\gamma,L}(n^{-1/2}).
\]
This proves \eqref{eq:small-imaginary-perturbation}. The formula with
$\kappa-\ell_n$ follows from the coefficient symmetry.
\end{proof}

\subsection*{Proof of Theorems A and B}

\begin{proof}
The one-parameter formula \eqref{eq:main-sqrt-r} is
Theorem~\ref{thm:complex-sqrt}, with $z=v$ and $\alpha=\delta$.
By Proposition~\ref{prop:pq-reduction},
\[
\log\left|
\frac{\pqbinom{n}{\kappa+\ell}{p,q}}
     {\pqbinom{n}{\kappa}{p,q}}
\right|
=
\log\left|
\frac{\pqbinom{n}{\kappa+\ell}{r}}{\pqbinom{n}{\kappa}{r}}
\right|
+\frac{u}{n}\ell^2.
\]
Since $x=\ell/\sqrt n$, this identity and
\eqref{eq:main-sqrt-r} give \eqref{eq:main-sqrt-pq}. The identity
$C_n(v,0)=A(v)/2$ follows from
\[
\frac{v}{e^{v/2}-1}
=
\frac{v}{2}\left(\coth\frac{v}{4}-1\right).
\]
Finally, \eqref{eq:main-Cn-explicit} follows by multiplying the numerator
and denominator in \eqref{eq:main-Cn} by
\[
e^{v/2-i\delta n^{1/4}/2}-1
\]
and taking real parts. This proves Theorem A.

Theorem B is Proposition~\ref{prop:small-imaginary-perturbation}; equation
\eqref{eq:intro-small-imaginary-perturbation} is the same as
\eqref{eq:small-imaginary-perturbation}.
\end{proof}

\section{Conclusion and future work}\label{sec:conclusion}

We have proved local asymptotic formulae for centred modulus ratios with
complex parameters. For an imaginary perturbation of order
$n^{-3/4}$, Theorem~A gives an expansion on the $\sqrt n$ range with an
$n$-dependent quadratic coefficient. Theorem~\ref{thm:complex-middle}
gives a uniform logarithmic--dilogarithmic approximation on the larger
$n^{3/4}$ range. At the real critical value, an imaginary perturbation of
order $n^{-5/4}$ changes the limiting quadratic coefficient but leaves the
quartic coefficient unchanged. This occurs even though
$\ell\Imag t\to0$ on that range.

One continuation is to obtain summation estimates for generating polynomials
without assuming positivity. In the positive real case, such arguments rely on
a uniform tail estimate. For complex coefficients, one also has to control
cancellation away from the centre.

A separate problem is to locate maxima of the coefficient moduli. The exact
adjacent ratio is
\[
\frac{\left|\pqbinom{n}{k+1}{p,q}\right|}
     {\left|\pqbinom{n}{k}{p,q}\right|}
=
|p|^{\,n-2k-1}
\frac{|1-r^{n-k}|}{|1-r^{k+1}|},
\qquad r=\frac qp.
\]
An interior local maximum at $k$ satisfies $|\rho_{k-1}(p,q)|\ge1$ and $|\rho_k(p,q)|\le1$. Deriving an asymptotic formula
for that crossing when the maxima move by order $n^{3/4}$ requires a separate
analysis; it does not follow from the bounded-$x$ expansion in Theorem A.

A further problem is to extend the estimates to
\[
t=\frac{v_0}{n^b}+i\frac{\delta_0}{n^a},
\qquad v_0>0.
\]
The exact centred identities remain valid for $\Real t>0$. A classification
of the possible limits must account separately for
\[
nt=v_0n^{1-b}+i\delta_0n^{1-a}
\]
and for $\ell t$ on the chosen coefficient range. As Theorem~B shows,
the condition $\ell\Imag t\to0$ alone does not determine whether the
imaginary perturbation contributes to the leading modulus asymptotic.

\section*{Declarations}

\subsection*{Funding}

The first-named author was partially funded by the Swedish Research Council under grant agreement no.~2025-05053 and by the Swedish Energy Agency under project no.~P2025-04323. The second-named author was partially
funded by the National Science Centre, Poland, under the Weave-UNISONO programme, grant no.~UMO-2025/07/Y/ST1/00146.

\subsection*{Competing interests}
The authors have no competing interests to declare that are relevant to the content of this article.

\subsection*{Author contributions}

All authors contributed equally to the mathematical development of the
work and to the preparation of the manuscript. 

\subsection*{Data availability}
No datasets were generated or analysed during the current study.


\begin{thebibliography}{99}
\bibitem{AhagCzyzLundow2024}
P. {\AA}hag, R. Czy{\.z}, and P. H. Lundow,
``On a generalised Lambert $W$ branch transition function arising from
$p,q$-binomial coefficients,''
\emph{Appl. Math. Comput.} \textbf{462} (2024), 128347;
\url{https://doi.org/10.1016/j.amc.2023.128347}.

\bibitem{AhagCzyzLundow2025}
P. {\AA}hag, R. Czy{\.z}, and P. H. Lundow,
``Complex branches of a generalized Lambert $W$ function arising from
$p,q$-binomial coefficients,''
\emph{Ann. Polon. Math.} \textbf{134} (2025), no.~1, 1--18;
\url{https://doi.org/10.4064/ap240830-1-7}.


\bibitem{ArabiArdehaliRosengren2026}
A. Arabi Ardehali and H. Rosengren,
``A new product formula for $(z;q)_\infty$, with applications to asymptotics,''
\emph{Constr. Approx.} (2026), published online;
\url{https://doi.org/10.1007/s00365-026-09783-2}.


\bibitem{AtakishiyevNagiyev1994}
N. M. Atakishiyev and Sh. M. Nagiyev,
``On the Rogers--Szeg\H{o} polynomials,''
\emph{J. Phys. A: Math. Gen.} \textbf{27} (1994), no.~17, L611--L615;
\url{https://doi.org/10.1088/0305-4470/27/17/003}.

\bibitem{Corcino2008}
R. B. Corcino,
``On $p,q$-binomial coefficients,''
\emph{Integers} \textbf{8} (2008), A29, 16 pp.

\bibitem{Dhand2014}
V. Dhand,
``A combinatorial proof of strict unimodality for $q$-binomial coefficients,''
\emph{Discrete Math.} \textbf{335} (2014), 20--24;
\url{https://doi.org/10.1016/j.disc.2014.07.001}.

\bibitem{GasperRahman2004}
G. Gasper and M. Rahman,
\emph{Basic Hypergeometric Series},
2nd ed., Encyclopedia of Mathematics and its Applications, vol.~96,
Cambridge University Press, Cambridge, 2004;
\url{https://doi.org/10.1017/CBO9780511526251}.

\bibitem{JagannathanSridhar2010}
R. Jagannathan and R. Sridhar,
``$(p,q)$-Rogers--Szeg\H{o} polynomial and the $(p,q)$-oscillator,''
in \emph{The Legacy of Alladi Ramakrishnan in the Mathematical Sciences},
K. Alladi, J. R. Klauder, and C. R. Rao (eds.),
Springer, New York, 2010, 491--501;
\url{https://doi.org/10.1007/978-1-4419-6263-8_29}.

\bibitem{LubinskySaff1987}
D. S. Lubinsky and E. B. Saff,
``Convergence of Pad\'e approximants of partial theta functions and the
Rogers--Szeg\H{o} polynomials,''
\emph{Constr. Approx.} \textbf{3} (1987), 331--361;
\url{https://doi.org/10.1007/BF01890574}.

\bibitem{LundowRosengren2010}
P. H. Lundow and A. Rosengren,
``On the $p,q$-binomial distribution and the Ising model,''
\emph{Philos. Mag.} \textbf{90} (2010), no.~24, 3313--3353;
\url{https://doi.org/10.1080/14786435.2010.484406}.

\bibitem{LundowRosengren2013}
P. H. Lundow and A. Rosengren,
``The $p,q$-binomial distribution applied to the 5d Ising model,''
\emph{Philos. Mag.} \textbf{93} (2013), no.~14, 1755--1770;
\url{https://doi.org/10.1080/14786435.2012.750770}.

\bibitem{DLMF17}
NIST Digital Library of Mathematical Functions,
\emph{\S17.2(ii) $q$-Binomial Coefficients},
National Institute of Standards and Technology,
\url{https://dlmf.nist.gov/17.2#ii}.

\bibitem{OHara1990}
K. M. O'Hara,
``Unimodality of Gaussian coefficients: a constructive proof,''
\emph{J. Combin. Theory Ser. A} \textbf{53} (1990), no.~1, 29--52;
\url{https://doi.org/10.1016/0097-3165(90)90018-R}.

\bibitem{PakPanova2013}
I. Pak and G. Panova,
``Strict unimodality of $q$-binomial coefficients,''
\emph{C. R. Math. Acad. Sci. Paris} \textbf{351} (2013), no.~11--12, 415--418;
\url{https://doi.org/10.1016/j.crma.2013.06.008}.

\bibitem{StanleyEC1}
R. P. Stanley,
\emph{Enumerative Combinatorics. Vol.~1},
2nd ed., Cambridge Studies in Advanced Mathematics, vol.~49,
Cambridge University Press, Cambridge, 2012;
\url{https://doi.org/10.1017/CBO9781139058520}.

\bibitem{SuWang2012}
X.-T. Su and Y. Wang,
``Proof of a conjecture of Lundow and Rosengren on the bimodality of
$p,q$-binomial coefficients,''
\emph{J. Math. Anal. Appl.} \textbf{391} (2012), no.~2, 653--656;
\url{https://doi.org/10.1016/j.jmaa.2012.02.049}.

\bibitem{Szego1926}
G. Szeg\H{o},
``Ein Beitrag zur Theorie der Thetafunktionen,''
\emph{Sitzungsberichte der Preussischen Akademie der Wissenschaften,
Physikalisch-Mathematische Klasse} (1926), 242--252.

\end{thebibliography}
\end{document}